\documentclass[11pt]{amsart}
\usepackage{amssymb}
\usepackage{tikz}
\usetikzlibrary{calc}
\usetikzlibrary{decorations.markings}

\newtheorem{lemma}{Lemma}
\newtheorem{theorem}{Theorem}

\theoremstyle{definition}

\newtheorem{remark}{Remark}

\newcommand{\C}{\mathbb{C}}
\newcommand{\R}{\mathbb{R}}
\newcommand{\Z}{\mathbb{Z}}
\newcommand{\bH}{\mathbb{H}}
\newcommand{\tei}{\mathcal{T}}
\newcommand{\ttei}{\widetilde{\mathcal{T}}}
\newcommand{\oscc}{\mathcal{S}}

\begin{document}

\title{Penner coordinates for closed surfaces}
\author{Rinat Kashaev}
\address{University of Geneva\\
2-4 rue du Li\`evre, Case postale 64\\
 1211 Gen\`eve 4, Suisse}
\email{rinat.kashaev@unige.ch}

\thanks{Supported in part by Swiss National Science Foundation}

\begin{abstract}
Penner coordinates are extended to the Teichm\"uller spaces of oriented closed surfaces.
\end{abstract}

\date{March 2, 2014}
\maketitle

\section{Introduction}
Penner coordinates in decorated Teichm\"uller spaces of punctured surfaces \cite{MR919235,Penner2012} are distinguished by the following two remarkable properties: 
\begin{enumerate}
\item the mapping class group action is rational;
\item the Weil--Petersson symplectic form is given explicitly by a simple formula.
\end{enumerate}  
Due to these properties, quantum theory of Teichm\"uller spaces has been successfully developed in
\cite{MR1607296,MR1737362} which resulted in construction of a one-parameter family of unitary projective mapping class group representations in infinite dimensional Hilbert spaces. For the fundamental groups of punctured surfaces, generalizations of Penner coordinates were constructed for the moduli spaces of faithful $SL(2,\C)$-representations in \cite{MR2078900} and for the moduli spaces of irreducible but not necessarily faithful  $PSL(2,\R)$-representations in \cite{MR2122727}. In this paper, we extend Penner coordinates to the Teichm\"uller spaces of oriented closed surfaces of genus $g>1$.

Let $S$ be a closed oriented surface of genus $g>1$, and
let 
\begin{equation}
R_k\subset \operatorname{Hom}(\pi_1,PSL(2,\R)),\quad \pi_1\equiv\pi_1( S,x_0),
\end{equation} be the connected component of representations of Euler number $k\in \Z$ with $|k|\le2g-2$. According to the result of Goldman \cite{MR952283}, the component $R_{2-2g}$ corresponds to discrete faithful representations, so that one has a principal $PSL(2,\R)$-fibre bundle over the Teichm\"uller space $\tei\equiv\tei(S)$
\begin{equation}
p\colon R_{2-2g}\to\tei.
\end{equation}
Denoting by $\Omega$ the space of all horocycles in the hyperbolic plane $\bH^2$, we consider the associated fibre bundle 
\begin{equation}
\phi\colon \ttei\to\tei,\quad \ttei\equiv R_{2-2g}\times_{PSL(2,\R)}\Omega,
\end{equation}
as a substitute  for Penner's decorated Teichm\"uller space in the case of closed surfaces. We define the $\lambda$-distance
\begin{equation}
\lambda\colon\Omega\times\Omega\to\R_{\ge0}
\end{equation}
as follows. If $h,h'\in\Omega$ are based on distinct points of $\partial \bH^2$, then $\lambda(h,h')$ is the hyperbolic length of the horocyclic segment between tangent points of a horocycle tangent simultaneously  to both $h$ and $h'$, and we define $\lambda(h,h')=0$ if $h$ and $h'$ are based on one and the same point of $\partial \bH^2$.

To any $\alpha\in\pi_1\setminus\{1\}$, we associate a function
\begin{equation}
\lambda_\alpha\colon\ttei\to\R_{\ge0}, \quad [\rho,h]\mapsto \lambda(\rho(\alpha)h,h).
\end{equation}
It is easily checked that
\begin{equation}
\lambda_\alpha=\lambda_{\alpha^{-1}}.
\end{equation}
The set
$\lambda_\alpha^{-1}(0)$
is a sub-bundle of $\ttei$ with the fibers homemorphic  to $\R\sqcup\R$. Moreover, one has 
\begin{equation}\label{eq:empty}
\alpha\ne\beta\Rightarrow\lambda_\alpha^{-1}(0)\cap\lambda_\beta^{-1}(0)=\emptyset.
\end{equation}

For any subset $A\subset\pi_1\setminus\{1\}$, we associate the subset
\begin{equation}
\ttei_A\equiv \cap_{\alpha\in A}\lambda^{-1}_\alpha(\R_{>0})
\end{equation}
 together with a function
 \begin{equation}
J_A\colon \ttei_A\to \R_{>0}^A,\quad J_A(x)(\alpha)=\lambda_\alpha(x),\quad \forall x\in\ttei_A,\ \forall\alpha\in A.
\end{equation}
In what follows, for any cellular complex $X$, we will denote by $X_i$ the set of its $i$-dimensional cells.

We define a \emph{triangulation} of $(S,x_0)$ as a cellular decomposition with only one vertex at $x_0$ and where all 2-cells are triangles. We denote by $\Delta\equiv \Delta( S,x_0)$ the set of  all triangulations of $(S,x_0)$. In principle, the characteristic maps induce orientations on all edges of a triangulation, but we will ignore this part of the information from the cellular structure.

 For any $\tau\in\Delta$, to any edge $e\in\tau_1$ there correspond two mutually inverse elements $\gamma^{\pm1}\in\pi_1$. By abuse of notation, we identify $e$ with any of the functions $\lambda_{\gamma^{\pm1}}$:
 \begin{equation}
e\equiv\lambda_{\gamma^{\pm1}}\colon \ttei\to\R_{>0}.
\end{equation}

For any $\tau\in\Delta$ and $e\in\tau_1$, we denote by $\tau^e$ the triangulation obtained by the diagonal flip at $e$, with the flipped edge being denoted as $e_\tau$:
   \begin{equation}
   \tau\ni\quad
 \begin{tikzpicture}[scale=1.5,baseline=-3]
\filldraw[color=gray!10] (0:1cm)--(90:1 cm)--(180:1cm)--(-90:1cm)--cycle;
\draw[auto] (90:1 cm) to node {$e$} (-90:1cm) ;
\filldraw (0:1cm) circle (.5pt)--(90:1 cm) circle (.5pt)--(180:1 cm) circle (.5pt)--(-90:1cm) circle (.5pt)--(0:1cm);
 \end{tikzpicture}\quad\rightsquigarrow\quad
 \begin{tikzpicture}[scale=1.5,baseline=-3]
\filldraw[color=gray!10] (0:1cm)--(90:1 cm)--(180:1cm)--(-90:1cm)--cycle;
\draw[auto](180:1 cm) to node {$e_\tau$} (0:1cm) ;
\filldraw (0:1cm) circle (.5pt)--(90:1 cm) circle (.5pt)--(180:1 cm) circle (.5pt)--(-90:1cm) circle (.5pt)--(0:1cm);
 \end{tikzpicture}\quad
 \in \tau^e
\end{equation}
It is easily shown that for any $\tau\in\Delta$, one has a finite covering
\begin{equation}
\ttei=\ttei_{\tau_1}\cup(\cup_{e\in\tau_1}\ttei_{\tau^e_1}).
\end{equation}
Our first result gives a realization of $\ttei_{\tau_1}$ as an algebraic subset of co-dimension one in $\R_{>0}^{\tau_1}$. In more precise terms, the result follows.

To any pair $(\tau,t)$ with $\tau\in\Delta$ and $t\in\tau_2$, we associate a function
\begin{equation}
\psi_{\tau,t}\colon\R_{>0}^{\tau_1}\to\R,\quad f\mapsto \sum_{t'\in\tau_2}\epsilon_t(t')\frac{a^2+b^2+c^2}{abc},
\end{equation}
where $a,b,c$ are the values of $f$ on three sides of $t'$, while the function 
\begin{equation}\label{eq:sign}
\epsilon_t\colon\tau_2\to\{-1,1\}
\end{equation}
takes the value $-1$ on $t$ and the value $1$ on all other triangles. We remark that
\begin{equation}\label{eq:inters}
t\ne t'\Rightarrow\psi_{\tau,t}^{-1}(0)\cap\psi_{\tau,t'}^{-1}(0)=\emptyset.
\end{equation}
We also define
\begin{equation}
\psi_\tau\equiv\prod_{t\in\tau_2}\psi_{\tau,t}.
\end{equation}
\begin{theorem}\label{thm}
 For any $\tau\in\Delta$, the map $J_{\tau_1}\colon\ttei_{\tau_1}\to\R_{>0}^{\tau_1}$ is an embedding with the image $\psi_\tau^{-1}(0)=\sqcup_{t\in\tau_2}\psi_{\tau,t}^{-1}(0)$.
   \end{theorem} 
   
\begin{remark}
 The transition functions $J_{\tau_1}\circ J_{\tau^e_1}^{-1}$ on the overlaps $\ttei_{\tau_1}\cap\ttei_{\tau_1^e}$ are given by  the signed Ptolemy transformation of \cite{MR2122727} (Proposition~4) with the sign function being given by \eqref{eq:sign}.  This is because  the inverse map $J_{\tau_1}^{-1}$ described in Section~\ref{sec2} is based on the same combinatorial rules as those of \cite{MR2122727}.
\end{remark}
   Let $\oscc\equiv\oscc(S)$   be the set of homotopy classes of essential simple closed curves in $S$, and $\Delta^\alpha\subset\Delta$ the set of triangulations of the form $\tau^\alpha$ with $\tau$ having an edge representing $\alpha$. 
From \eqref{eq:empty}, it is easily seen that
\begin{equation}
 \lambda_\alpha^{-1}(0)\subset\ttei_{\tau_1},\quad \forall \tau\in\Delta^\alpha.
\end{equation}
Our second result gives explicit coordinatization of the sub-bundles $\lambda_\alpha^{-1}(0)$ together with the explicit $\R_{>0}$-action along the fibers. The result follows.

For $\alpha\in\oscc$, 
let 
\begin{equation}
\ell_\alpha\colon\tei\to\R_{>0}
\end{equation}
be the hyperbolic length of the geodesic in the homotopy class of $\alpha$.
Any $\tau\in\Delta^\alpha$ has a distinguished edge $\alpha_\tau$. Let $\tau_\alpha$ be the quadrilateral having $\alpha_\tau$ as its diagonal. 
\begin{theorem}\label{thm:2} Let $\alpha\in\oscc$ and $\tau\in\Delta^\alpha$. Then
\begin{description}
\item[(i)] one has the inclusion $J_{\tau_1}(\lambda_\alpha^{-1}(0))\subset\cup_{t\in (\tau_\alpha)_2}\psi_{\tau,t}^{-1}(0)$;
\item[(ii)] for any $t\in(\tau_\alpha)_2$, the map 
\begin{multline}
L_{\alpha,\tau,t}\colon\ttei(\alpha,t)\equiv \lambda_\alpha^{-1}(0)\cap (\psi_{\tau,t}\circ J_{\tau_1})^{-1}(0)\to \R_{>0}\times\R^{\tau_1\setminus t_1}_{>0}\\
m\mapsto (\ell_\alpha(\phi(m)),J_{\tau_1\setminus t_1}(m))
\end{multline}
is a homeomorphism;
\item[(iii)] For any $d\in\R_{>0}$ one has the following equivalence
\begin{multline}\label{eq:equiv}
\phi(m)=\phi(m')
\Leftrightarrow \exists\  c\in\R_{>0}\colon J_{\tau_1\setminus t_1}(m)=c\,J_{\tau_1\setminus t_1}(m'),\\\forall m,m'\in(\ell_\alpha\circ\phi)^{-1}(d)\cap \ttei(\alpha,t).
\end{multline}
\end{description}
\end{theorem}
\begin{remark}
 The space  $\ttei(\alpha,t)$ in Theorem~\ref{thm:2} is a connected component of $\lambda_\alpha^{-1}(0)$.  It can also be singled out by fixing an orientation on $\alpha$, and considering only the classes $[\rho,h]$, with $h$  based on the attracting fixed point of $\rho(\alpha)$.
\end{remark}
The paper is organized as follows.  In Section~\ref{sec1} we collect necessary material on the group $PSL(2,\R)$ and we prove the important Lemmas~\ref{lem:1}--\ref{lem:3}. Sections~\ref{sec2} and \ref{sec3} contain proofs of  Theorems~\ref{thm} and \ref{thm:2} respectively.

\subsection*{Acknowledgements} This work is supported in part by Swiss National Science Foundation. Some of the results were reported at the Oberwolfach workshop ``New Trends in Teichm\"uller Theory and Mapping Class Groups" in February 2014. I would like to thank the participants of this workshop for useful and helpful discussions, especially J.~Andersen, M.~Burger, L.~Chekhov, V.~Fock,  L.~Funar, W.~Goldman, N.~Kawazumi, F.~Luo, G.~Masbaum, N.~Reshetikhin, R.~van der Veen, A.~Virelizier, A.~Wienhard. 
\section{Factorization in $SL(2)$}\label{sec1}
The \emph{matrix coefficients } of the group $SL(2,\R)$ are the mappings
\begin{equation}
a,b,c,d\colon SL(2,\R)\to \R
\end{equation}
such that
\begin{equation}
g=
\begin{pmatrix}
 a(g)&b(g)\\c(g)&d(g)
\end{pmatrix},\quad \forall g\in SL(2,\R).
\end{equation}
We fix two group embeddings 
\begin{equation}
u,v\colon \R\to SL(2,\R)
\end{equation}
defined by
\begin{equation}
u(x)=
\begin{pmatrix}
 1&x\\
 0&1
\end{pmatrix},\quad
v(x)=
\begin{pmatrix}
 1&0\\
 x&1
\end{pmatrix}.
\end{equation}
It is easily verified that an element $g\in SL(2,\R)$  with nonzero left lower coefficient, i.e. $c(g)\ne0$, is uniquely factorized  as follows:
\begin{equation}\label{eq:f1}
g=u(x)v(y)u(z)=
\begin{pmatrix}
 1+xy&x+z+xyz\\
 y&1+yz
\end{pmatrix}
\end{equation}
where
\begin{equation}\label{eq:f2}
 x=(a(g)-1)c(g)^{-1},\quad y=c(g),\quad z=c(g)^{-1}(d(g)-1).
\end{equation}
Moreover, for any $(x,y,z)\in\R^3_{>0}$, there exists a unique triple $(x',y',z')\in\R^3_{>0}$ such that
 \begin{equation}
v(x)u(y)v(z)=u(z')v(y')u(x').
\end{equation}
Explicitly, we have
\begin{equation}\label{eq:elec-net}
(x',y',z')=\left((x+xyz+z)^{-1}xy,x+xyz+z,yz(x+xyz+z)^{-1}\right)
\end{equation}
\begin{remark}
The map
\begin{equation}
R\colon \R^3_{>0}\to\R^3_{>0},\quad (x,y,z)\mapsto (x',y',z')
\end{equation}
is an involution which solves the set-theoretical tetrahedron equation
\begin{equation}
R_{123}\circ R_{145}\circ R_{246}\circ R_{356}=R_{356}\circ R_{246}\circ R_{145}\circ R_{123}.
\end{equation}
This solution is related with the star-triangle  transformation in electrical networks \cite{MR1421683,MR1694046}.
\end{remark}
\subsection{Universal covering $\widetilde{SL}(2,\R)$}
Let us define two coordinate charts covering the group manifold of $PSL(2,\R)$. We define two open contractible sets
\begin{equation}
U_1\equiv PSL(2,\R)\setminus |a|^{-1}(0),\quad U_2\equiv PSL(2,\R)\setminus |b|^{-1}(0)
\end{equation}
together with the homeomorphisms
\begin{equation}
\varphi_j\colon U_j\to \bH^3,\ j\in\{1,2\},\quad
\varphi_1=\left(\frac{b}{a},\frac{c}{a}, |a|\right),\ \varphi_2=\left(\frac{a}{b},\frac{d}{b}, |b|\right).
\end{equation}
The intersection $U_1\cap U_2$ consists of two contractible components
\begin{equation}
U_1\cap U_2=U_{12}^+\sqcup U_{12}^-,\quad U_{12}^\pm\equiv\{\pm ab>0\}.
\end{equation}
Let
\begin{equation}
p\colon\widetilde{SL}(2,\R)\to PSL(2,\R)
\end{equation}
be the canonical projection from the universal covering space. We fix a group isomorphism
\begin{equation}
\Phi\colon \Z\to p^{-1}(\pm1)\simeq\pi_1(PSL(2,\R),\pm1)
\end{equation}
sending an integer $n$ to the homotopy class of the loop 
\begin{equation}
\omega_n\colon [0,1]\ni t\mapsto
\pm\begin{pmatrix}
 \cos (n\pi t)&\sin(n\pi t)\\
 -\sin(n\pi t)&\cos(n\pi t)
\end{pmatrix}.
\end{equation}
We remark that for $n=1$ we have the following inclusions
\begin{multline}\label{eq:gamma0incls}
\omega_1\left([0,1]\setminus\{1/2\}\right)\subset U_1,\quad\omega_1\left(]0,1[\right)\subset U_2,\\
\omega_1\left(]0,1/2[\right)\subset U_{12}^+,\quad \omega_1\left(]1/2,1[\right)\subset U_{12}^-.
\end{multline}
For any $x\in\R$, we fix the lifts $\widetilde{u}(x)$, $\widetilde{v}(x)$ to $\widetilde{SL}(2,\R)$ represented by the paths
\begin{equation}\label{eq:lift-uv}
\widetilde{u}(x),\ \widetilde{v}(x)\colon [0,1]\to SL(2,\R),\quad
\widetilde{u}(x)(t)=u(xt),\
\widetilde{v}(x)(t)=v(xt).
\end{equation}
\begin{lemma}\label{lem:1}
For any $ x\in\R_{>0}$, let $\gamma_x$ be the lift of the left hand side of the $SL(2,\R)$-identity
\begin{equation}\label{eq:diamond}
(v(-x)u(2/x))^2=-1
\end{equation}
obtained by using the lifts~\eqref{eq:lift-uv}. Then the path homotopy class of $\gamma_x$  is given by the class $\Phi(1)$.
\end{lemma}
\begin{proof}
By using \eqref{eq:lift-uv}, we have
\begin{multline}
\gamma_{x}\equiv(\widetilde{v}(-x)\widetilde{u}(2/x))^2\colon [0,1]\to SL(2,\R),\quad
t\mapsto 
(v(-xt)u(2t/x))^2\\=
\begin{pmatrix}
1&2t/x\\
-tx&1-2t^2
\end{pmatrix}^2=\begin{pmatrix}
1-2t^2&4t(1-t^2)/x\\
-2t(1-t^2)x&1-6t^2+4t^4
\end{pmatrix}
\end{multline}
which has the properties
\begin{multline}\label{eq:gdin}
\gamma_{x}\left([0,1]\setminus\{1/\sqrt{2}\}\right)\subset U_1,\quad\gamma_{x}\left(]0,1[\right)\subset U_2,\\
\gamma_{x}\left(]0,1/\sqrt{2}[\right)\subset U_{12}^{+},\quad \gamma_{x}\left(]1/\sqrt{2},1[\right)\subset U_{12}^{-}.
\end{multline}
By comparing \eqref{eq:gdin} with \eqref{eq:gamma0incls}, we conclude that the path homotopy class of $\gamma_{x}$ coincides with that of $\omega_1$. 
\end{proof}
\begin{lemma}\label{lem:2}
 For any $x=(x_1,x_2,x_3)\in\R^3_{>0}$, let $x'\in\R_{>0}^3$ be the unique point such that
\begin{equation}\label{eq:tetr}
v(x_1)u(x_2)v(x_3)u(-x_3')v(-x_2')u(-x_1')=1.
\end{equation}
Let $\alpha_{x}$ be the lift of the left hand side of \eqref{eq:tetr} obtained by using the lifts~\eqref{eq:lift-uv}. Then the path homotopy class of $\alpha_{x}$ is given by the class $\Phi\left(0\right)$.
\end{lemma}
\begin{proof}
We just remark that the one parameter family of loops $\{f_t=\alpha_{tx}\}_{t\in[0,1]}$ is a well defined path homotopy between $\alpha_{x}$ and  the constant path.
\end{proof}
\begin{lemma}\label{lem:3} For any $x=(x_1,x_2,x_3)\in\R_{>0}^3$, let $\bar x\in\R^3$ be the unique point such that
\begin{equation}\label{eq:hex-e}
 v(x_1)u(-\bar x_3)v(x_2)u(-\bar x_1)v(x_3)u(-\bar x_2)=\epsilon\in\{-1,1\}.
\end{equation}
Let $\beta_{x,\epsilon}$ be the lift of the left hand side of \eqref{eq:hex-e} obtained by using the lifts~\eqref{eq:lift-uv}. Then the path homotopy class of $\beta_{x,\epsilon}$ is given by the class $\Phi\left(-(3+\epsilon)/2\right)$.
\end{lemma}
\begin{proof}
Let $n\equiv\Phi^{-1}([\beta_{x,\epsilon}])$. We treat the different values of $\epsilon$ differently.

In the case $\epsilon=1$, let $i\in\{1,2,3\}$ be such that
\begin{equation}
x_i=\max(x_1,x_2,x_3).
\end{equation}
By making two substitutions 
\begin{equation}\label{eq:subs1}
v(x_i)\rightsquigarrow u(2/x_i)v(-x_i)u(2/x_i),\quad  u(-\bar x_i)\rightsquigarrow v(-2/\bar x_i)u(\bar x_i)v(-2/\bar x_i)
\end{equation}
 in the identity \eqref{eq:hex-e} and by using cyclic permutations of factors (which do not change the lift) and group homomorphism properties of the maps $u$ and $v$, we transform \eqref{eq:hex-e} into a case of identity~\eqref{eq:tetr}. By Lemma~\ref{lem:1}, each substitution in \eqref{eq:subs1} increases by one the integer associated to the homotopy class of its lift so that the class of the lift after these two substitutions is given by $\Phi(n+2)$. On the other hand, by Lemma~\ref{lem:2}, the same class is given by $\Phi(0)$. Thus,  we arrive at the conclusion that $n+2=0$.
 
 The case $\epsilon=-1$ is treated differently depening on existence or non-existence of Euclidean triangles with side lengths given by the components of $x$. 
 
 Assume first that there are no such triangles. It means that there exists an index $i\in\{1,2,3\}$ such that $x_i> x_j+x_k$, where $\{j,k\}=\{1,2,3\}\setminus\{i\}$. In that case, the substitution
 \begin{equation}\label{eq:subs2}
v(x_i)\rightsquigarrow u(2/x_i)v(-x_i)u(2/x_i)
\end{equation}
followed by cyclic permutations of factors and subsequent simplifications of products of $u$-terms transforms identity~\eqref{eq:hex-e} into a case of identity~\eqref{eq:tetr}. By a similar reasoning as above for the case with $\epsilon=1$, we conclude that $n+1=0$. 

Assume now that there exists an Euclidean triangle of non-zero area with side lengths given by the components of $x$. In that case, we choose arbitrary index $i\in\{1,2,3\}$ and make two substitutions 
 \begin{equation}\label{eq:subs3}
v(x_j)\rightsquigarrow u(2/x_j)v(-x_j)u(2/x_j),\quad v(x_k)\rightsquigarrow u(2/x_k)v(-x_k)u(2/x_k)
\end{equation}
with $\{j,k\}=\{1,2,3\}\setminus\{i\}$. Applying necessary cyclic permutations and simplifications of $u$-terms we arrive at an identity of the same form as \eqref{eq:hex-e} but with negated components $x_j$ and $x_k$ and with the accordingly modified point $\bar x$. In this new identity, the substitution
\begin{equation}\label{eq:subs4}
u(-\bar x_i)\rightsquigarrow v(-2/\bar x_i)u(\bar x_i)v(-2/\bar x_i),
\end{equation}
followed by cyclic permutations of factors and subsequent simplifications of products of $v$-terms,  brings it to a case of identity~\eqref{eq:tetr}. Under the fist two substitutions~\eqref{eq:subs3}, the integer $n$ is increased by two, while under the last substitution~\eqref{eq:subs4} it is reduced by one with the final value being $n+2-1=n+1$. Again, by Lemma~\ref{lem:2}, we conclude that $n+1=0$.

Finally, it remains the degenerate case where there exists $i\in\{1,2,3\}$ such that $x_i=x_j+x_k$, where
$\{j,k\}=\{1,2,3\}\setminus\{i\}$. That means that $\bar x_i=0$, which allows to reduce identity~\eqref{eq:hex-e} to the inverse of \eqref{eq:diamond}. Thus, by Lemma~\ref{lem:1}, $n=-1$ in this case as well.
\end{proof}

 \section{Proof of Theorem~\ref{thm}}\label{sec2}
 
 Let $\tau\in\Delta$. By cutting out a small open disk $D\subset S$ centered at $x_0$, we obtain a cellular decomposition $\bar\tau$ of $S'\equiv S\setminus D$, where the vertex set is given by the intersection points of edges of $\tau$ with the boundary of $S'$, with two types of edges: \emph{long edges} given by the remnants of the edges of $\tau$ and \emph{short edges} given by the boundary segments between the vertices, and hexagonal 2-cells given by truncated triangles of $\tau$. We assume that the short edges are canonically oriented through the counterclockwise orientation of the boundary of $D$. In what follows, we will abuse  the notation by identifying the long edges of $\bar\tau$ with the edges of $\tau$ and the hexagonal faces of $\bar\tau$ with the triangular  faces of $\tau$, and also we
 will think of the elements of $\tau_1$ and $\tau_2$ as functions on the set $\R_{>0}^{\tau_1}\times \{-1,1\}^{\tau_2}$ in the sense that
\begin{equation}
e\colon \R_{>0}^{\tau_1}\times \{-1,1\}^{\tau_2}\to \R_{>0}, \quad (f,\varepsilon)\mapsto f(e),\quad \forall e\in\tau_1,
\end{equation}
and 
\begin{equation}
t\colon \R_{>0}^{\tau_1}\times \{-1,1\}^{\tau_2}\to \{-1,1\}, \quad (f,\varepsilon)\mapsto \varepsilon(t),\quad \forall t\in\tau_2.
\end{equation}
Depending on the value of $t$, the triangle will be called \emph{positive} or \emph{negative}.

 Following \cite{MR2122727}, we start by constructing a map
 \begin{equation}
\xi_\tau\colon \R_{>0}^{\tau_1}\times \{-1,1\}^{\tau_2}\to \operatorname{Hom}(\pi_1(S',\bar\tau_0),PSL(2,\R))
\end{equation}
where $\pi_1(S',\bar\tau_0)$ is the fundamental groupoid of $S'$ with the vertex set $\bar\tau_0$. The construction is as follows.

To any long edge $e$, we associate the element $\pm w(e)$ where
 \begin{equation}
 w(e)\equiv
\begin{pmatrix}
 0&-e^{-1}\\e&0
\end{pmatrix}.
\end{equation}
As the element $\pm w(e)$ is of order two, our assignment is valid  for both orientations of $e$.
For any short edge $e'$ (with the clockwise orientation with respect to the center of the hexagon to which it belongs) we associate the element $\pm u\!\left(t\frac{a}{bc}\right)$, with $t$ being the unique hexagonal face having $e'$ as its side, $a$ is the long side of $t$ opposite to $e'$, while $b$ and $c$ are two other long sides of $t$. These assignments are illustrated in this picture:
\begin{equation}\label{pic:assign}
\begin{tikzpicture}[scale=1.5,baseline=-3,decoration={
markings,
mark=at position .5 with {\arrow{stealth};}}]
\filldraw[color=gray!10] (20:1cm)--(100:1 cm)--(140:1 cm)--(-140:1cm)--(-100:1 cm)--(-20:1 cm)--cycle;
\draw[auto,color=green,postaction={decorate}] (20:1cm) to node{$\pm u\!\left(t\frac{a}{bc}\right)$} (-20:1cm);
\draw[color=green,postaction={decorate}] (140:1 cm)--(100:1 cm); 
\draw[color=green,postaction={decorate}] (-100:1 cm)--(-140:1 cm); 
\filldraw[auto] (20:1cm) circle (.5pt) to node {$b$} (100:1 cm) circle (.5pt) (140:1 cm) circle (.5pt) to node {$a$} node[swap]{$\pm w(a)$} (-140:1 cm) circle (.5pt)(-100:1cm) circle (.5pt) to node {$c$} (-20:1 cm) circle (.5pt);
\node at (0,0){$t$};
 \end{tikzpicture}
 \end{equation}
 where long edges are drawn in black and short edges in green.
 It is straightforward to check that this assignment uniquely extends to a representation of $\pi_1(S',\bar\tau_0)$. For $(f,\varepsilon)\in\xi_\tau^{-1}(R_{2-2g})$, by taking the equivalence class of the pair $(\xi_\tau(f,\varepsilon), h_0)$, where $h_0$ is the horocycle based at $\infty$ and passing through $i\in\bH^2$, we obtain a map 
 \begin{equation}
 \tilde\xi_{\tau}\colon\xi_\tau^{-1}(R_{2-2g}) \to \ttei
 \end{equation}
which will be shown to be the inverse of $J_{\tau_1}$.
\subsection{Calculation of the Euler number}
In the case $g=w(a)$, the factorization~\eqref{eq:f1}, \eqref{eq:f2} takes the form
\begin{equation}
w(a)=u(-1/a)v(a)u(-1/a),\quad \forall a\in\R_{>0}.
\end{equation}
We implement this factorization combinatorially by transforming $\bar\tau$ to a new cellular complex $\tilde\tau$ as follows. 

We insert two new vertices in each short edge of $\bar\tau$ and connect them inside each hexagonal face by three oriented edges parallel to long edges, the orientations being counterclockwise with respect to the center of the hexagon.  We associate  to these edges the elements of $\pm v(a)$ where the argument $a$ is the  corresponding long edge. The following picture summarizes this subdivision:
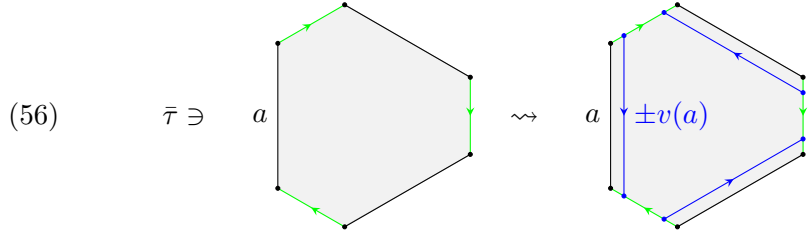
\begin{equation}
\bar\tau\ni\quad \begin{tikzpicture}[scale=1.5,baseline=-3,decoration={
markings,
mark=at position .5 with {\arrow{stealth};}}]
\filldraw[color=gray!10] (20:1cm)--(100:1 cm)--(140:1 cm)
--(-140:1cm)--(-100:1 cm)--(-20:1 cm)--cycle;
\draw[color=green,postaction={decorate}] (140:1 cm)--(100:1 cm);
\draw[color=green,postaction={decorate}] (-100:1 cm)--(-140:1 cm);
\draw[color=green,postaction={decorate}] (20:1 cm)--(-20:1 cm);
\filldraw (20:1cm) circle (.5pt)--(100:1 cm) circle (.5pt) (140:1 cm) circle (.5pt)
(-140:1 cm) circle (.5pt)(-100:1cm) circle (.5pt)--(-20:1 cm) circle (.5pt);
\draw[auto] (140:1cm) to node[swap] {$a$} (-140:1cm);
 \end{tikzpicture}\quad
\rightsquigarrow\quad
\begin{tikzpicture}[scale=1.5,baseline=-3,decoration={
markings,
mark=at position .5 with {\arrow{stealth};}}]
\filldraw[color=gray!10] (20:1cm)--(100:1 cm)--(140:1 cm)--
(-140:1cm)--(-100:1 cm)--(-20:1 cm)--cycle;
\draw[color=green,postaction={decorate}] (140:1 cm)--(100:1 cm);
\draw[color=green,postaction={decorate}] (-100:1 cm)--(-140:1 cm);
\draw[color=green,postaction={decorate}] (20:1 cm)--(-20:1 cm);
\filldraw (20:1cm) circle (.5pt)--(100:1 cm) circle (.5pt) (140:1 cm) circle (.5pt)
(-140:1 cm) circle (.5pt)(-100:1cm) circle (.5pt)--(-20:1 cm) circle (.5pt);
\draw[auto] (140:1cm) to node[swap] {$a$} (-140:1cm);
\filldraw[color=blue] ($(-20:1cm)! .2!(20:1cm)$) circle (.5pt) ($(-100:1cm)! .2!(-140:1cm)$) circle (.5pt)
($(100:1cm)! .8!(140:1cm)$) circle (.5pt) ($(-100:1cm)! .8!(-140:1cm)$) circle (.5pt)
 ($(-20:1cm)! .8!(20:1cm)$) circle (.5pt)($(100:1cm)! .2!(140:1cm)$) circle (.5pt);
\draw[color=blue,postaction={decorate}] ($(-100:1cm)! .2!(-140:1cm)$)--($(-20:1cm)! .2!(20:1cm)$)  ;
\draw[auto,color=blue,postaction={decorate}] ($(100:1cm)! .8!(140:1cm)$) to node {$\pm v(a)$} ($(-100:1cm)! .8!(-140:1cm)$);
\draw[color=blue,postaction={decorate}] ($(-20:1cm)! .8!(20:1cm)$)--($(100:1cm)! .2!(140:1cm)$);
 \end{tikzpicture}
 \end{equation}
 where the added vertices and edges are drawn in blue. In this subdivided complex,  the long edges of $\bar\tau$ will be called \emph{primary long edges} while the newly added edges will be called \emph{secondary long edges}. There are now two types of 2-cells: rectangular and hexagonal faces. Rectangular faces come naturally in pairs where each pair is associated with a unique primary long edge. Within each such pair, let us glue two rectangular faces along their common primary long sides and then erase the primary long edge as is described in this picture:
 
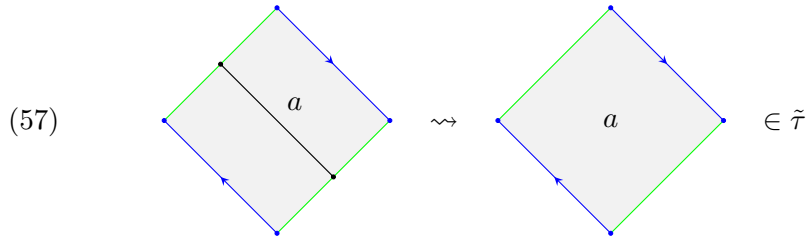
\begin{equation}
\begin{tikzpicture}[scale=1.5,baseline=-3,decoration={
markings,
mark=at position .5 with {\arrow{stealth};}}]
\filldraw[color=gray!10] (0:1cm)--(90:1cm)--(180:1cm)--(-90:1cm)--cycle;
\draw[color=green] (0:1cm)--(-90:1cm) (180:1cm)--(90:1cm);
\draw[color=blue,postaction={decorate}] (-90:1cm)--(180:1cm);
\draw[color=blue,postaction={decorate}] (90:1cm)--(0:1cm);
\filldraw[auto] ($(180:1cm)! .5!(90:1cm)$)  circle (.5pt) to node {$a$} ($(-90:1cm)! .5!(0:1cm)$)  circle (.5pt);
\filldraw[color=blue] (0:1cm) circle (.5pt)  (90:1cm) circle (.5pt)  (180:1cm) circle (.5pt) (-90:1cm) circle (.5pt);
 \end{tikzpicture}\quad
\rightsquigarrow\quad
\begin{tikzpicture}[scale=1.5,baseline=-3,decoration={
markings,
mark=at position .5 with {\arrow{stealth};}}]
\filldraw[color=gray!10] (0:1cm)--(90:1cm)--(180:1cm)--(-90:1cm)--cycle;
\draw[color=green] (0:1cm)--(-90:1cm) (180:1cm)--(90:1cm);
\draw[color=blue,postaction={decorate}] (-90:1cm)--(180:1cm);
\draw[color=blue,postaction={decorate}] (90:1cm)--(0:1cm);
\node at (0,0){$a$};
\filldraw[color=blue] (0:1cm) circle (.5pt)  (90:1cm) circle (.5pt)  (180:1cm) circle (.5pt) (-90:1cm) circle (.5pt);
 \end{tikzpicture}
 \quad\in\tilde\tau
 \end{equation}
 The result is our transformed complex $\tilde\tau$ which has rectangular faces which are in bijection with the edges of $\tau$ and  and hexagonal faces which are in bijection with the triangles of $\tau$. By taking into account the group elements associated with the edges, each rectangular face of $\tilde\tau$ corresponds to a case of the 
 relation~\eqref{eq:diamond} where variable $x$ is given by the positive number associated with the corresponding primary edge, while each hexagonal face of $\tilde\tau$ corresponds to a case  of the relation~\eqref{eq:hex-e} where three components of the vector $x$ are given by three positive numbers associated with three (primary) long edges around the hexagon and $\epsilon$ being given by the value of the corresponding function $t\in\tau_2$. Under the lifts~\eqref{eq:lift-uv}, each face contributes an integer to the Euler number of the representation. The contributions are controlled by Lemma~\ref{lem:1} for the rectangular faces and by Lemma~\ref{lem:3} for the hexagonal faces. Let $N_-$ and $N_+$ be the numbers of negative and positive triangles respectively. By Lemma~\ref{lem:1}, the total contribution from all rectangular faces is the number of edges of $\tau$, i.e. $6g-3$, while, by Lemma~\ref{lem:3}, the total contribution of the hexagonal faces is $-N_--2N_+$. Thus, the Euler number of $\xi_\tau(f,\varepsilon)$ is calculated as follows:
 \begin{equation}\label{eq:euler}
e(\xi_\tau(f,\varepsilon))=6g-3-N_- -2N_+=1+N_--2g,
\end{equation}
where we have taken into account the equality $N_-+N_+=4g-2$. Thus, the subset $\xi_\tau^{-1}(R_{2-2g})$ is completely characterized by the condition $N_-=1$, i.e. that there is only one negative triangle, and the vanishing condition for the boundary  holonomy which takes the form
\begin{equation}
\sum_{t\in\tau_2}t\frac{p_t}{q_t}=0,\quad p_t\equiv\sum_{e\in t_1}e^2,\quad q_t\equiv\prod_{e\in t_1}e.
\end{equation}
These two conditions, in their turn, are equivalent to a single vanishing condition $\psi_\tau=0$ in $\R_{>0}^{\tau_1}$ with
\begin{equation}
\psi_\tau\equiv\prod_{t\in\tau_2}\psi_{\tau,t},\quad \psi_{\tau,t}\equiv-\frac{p_t}{q_t}+\sum_{s\in\tau_2\setminus\{t\}}\frac{p_s}{q_s}.
\end{equation}
Thus, by taking into account the intersection properties \eqref{eq:inters}, we conclude that we have a natural identification
\begin{equation}
\xi_\tau^{-1}(R_{2-2g})\simeq\psi_\tau^{-1}(0).
\end{equation}

\subsection{Verification of the equality $\tilde\xi_{\tau}=J_{\tau_1}^{-1}$}
The equality 
\begin{equation}
J_{\tau_1}(\tilde \xi_{\tau}(f))=f,\quad\forall f\in \psi_\tau^{-1}(0),
\end{equation}
is checked straightforwardly, while the reverse equality 
\begin{equation}
\tilde \xi_{\tau}(J_{\tau_1}(m))=m,\quad\forall m\in\ttei_{\tau_1},
\end{equation}
 is checked by choosing a representative $(\rho,h_0)$ of $m$ where $h_0$ is the horocycle based at $\infty$ and passing through $i\in\bH^2$. For such a representative, the representation $\rho$ is defined uniquely up to conjugation by elements of the group of upper triangular unipotent  matrices (the  stabilizer subgroup for $h_0$).  

For each triangle $t$ of $\tau$, we take three $PSL(2,\R)$-elements representing its three oriented sides, the orientations being chosen cyclically in the counter-clockwise direction with respect to the center of $t$, and we choose the $SL(2,\R)$-representatives of those elements which have positive left lower matrix elements, i.e. $c>0$. We associate to $t$ the sign of the cyclic product of those representatives along the boundary of $t$, thus obtaining a map 
\begin{equation}\label{eq:epsm}
\varepsilon_{m}\colon\tau_2\to\{-1,1\}.
\end{equation}
 We also apply the factorization formula~\eqref{eq:f1} to each of those representatives thus realizing $\rho$ as a representation of the form $\xi_\tau(f,\varepsilon_m)$. The Euler number calculation~\eqref{eq:euler} implies that $\varepsilon_{m}=\epsilon_t$ for some $t\in\tau_2$.
 
 \subsection{Signed Ptolemy transformation}
The transition functions $J_{\tau_1}\circ J_{\tau^e_1}^{-1}$ on the overlaps $\ttei_{\tau_1}\cap\ttei_{\tau_1^e}$ are given by the same signed Ptolemy transformation as in \cite{MR2122727} (Proposition~4). This is a consequence of the definition of the map $\xi_\tau$ based on the same rules~\eqref{pic:assign} of assigning group elements on the edges of truncated triangulations. Below, following \cite{MR2122727}, we derive the signed Ptolemy transformation.  

If two pairs with one and the same horocyclic components represent the same point in $\ttei$, then the representation components are conjugated by an element of the stabilizer subgroup of the common horocycle which is a parabolic subgroup isomorphic to $\R$. That means that the parallel transport operators on the short edges  of the truncated triangulations, being in the same subgroup, must be the same independently of the triangulation. By using the rules~\eqref{pic:assign}, we can write, for example, the equality for the parallel transport operators along the short edge between the long edges $a$ and $b$ on two sides of this picture
\begin{equation}
 \tau\ni\quad\begin{tikzpicture}[scale=1.5,baseline=-3,decoration={
markings,
mark=at position .5 with {\arrow{stealth};}}]
\filldraw[color=gray!10] (0:1cm)--(45:1 cm)--(90:1cm)--(135:1cm)--(180:1cm)--(-135:1 cm)--(-90:1cm)--(-45:1cm)--cycle;
\draw[color=green,postaction={decorate}] (180:1 cm)--(135:1 cm);
\draw[color=green,postaction={decorate}] (0:1 cm)--(-45:1 cm);
\draw[color=green,postaction={decorate}] (90:1 cm)--($(90:1cm)! .5!(45:1cm)$);
\draw[color=green,postaction={decorate}]($(90:1cm)! .5!(45:1cm)$)-- (45:1 cm);
\draw[color=green,postaction={decorate}] (-90:1 cm)--($(-90:1cm)! .5!(-135:1cm)$);
\draw[color=green,postaction={decorate}]($(-90:1cm)! .5!(-135:1cm)$)-- (-135:1 cm);
\node at (157.5:.5cm){$\alpha$};
\node at (-22.5:.5cm){$\beta$};
\node (c) at (0,0){$e$};
\draw [auto](0:1 cm)to node[swap] {$b$}(45:1cm) 
(90:1cm)to node[swap] {$a$}(135:1cm)(180:1cm)to node [swap]{$d$}(-135:1cm)(-90:1cm)to node[swap] {$c$}(-45:1cm) ;
\filldraw (0:1cm) circle (.5pt)(45:1 cm) circle (.5pt)(90:1 cm) circle (.5pt)(135:1cm) circle (.5pt)(180:1cm) circle (.5pt)(-135:1 cm) circle (.5pt)(-90:1 cm) circle (.5pt)(-45:1cm) circle (.5pt);
\filldraw($(-90:1cm)! .5!(-135:1cm)$) circle (.5pt) --(c)-- ($(90:1cm)! .5!(45:1cm)$) circle (.5pt);
 \end{tikzpicture}
 \quad\leftrightsquigarrow\quad
  \begin{tikzpicture}[scale=1.5,baseline=-3,decoration={
markings,
mark=at position .5 with {\arrow{stealth};}}]
\filldraw[color=gray!10] (0:1cm)--(45:1 cm)--(90:1cm)--(135:1cm)--(180:1cm)--(-135:1 cm)--(-90:1cm)--(-45:1cm)--cycle;
\draw[color=green,postaction={decorate}] (90:1 cm)--(45:1 cm);
\draw[color=green,postaction={decorate}] (-90:1 cm)--(-135:1 cm);
\draw[color=green,postaction={decorate}] (0:1 cm)--($(0:1cm)! .5!(-45:1cm)$);
\draw[color=green,postaction={decorate}]($(0:1cm)! .5!(-45:1cm)$)-- (-45:1 cm);
\draw[color=green,postaction={decorate}] (180:1 cm)--($(180:1cm)! .5!(135:1cm)$);
\draw[color=green,postaction={decorate}]($(180:1cm)! .5!(135:1cm)$)-- (135:1 cm);
\node at (67.5:.5cm){$\gamma$};
\node at (-112.5:.5cm){$\delta$};
\node (c) at (0,0){$f$};
\draw [auto](0:1 cm)to node[swap] {$b$}(45:1cm) 
(90:1cm)to node[swap] {$a$}(135:1cm)(180:1cm)to node [swap]{$d$}(-135:1cm)(-90:1cm)to node[swap] {$c$}(-45:1cm) ;
\filldraw (0:1cm) circle (.5pt)(45:1 cm) circle (.5pt)(90:1 cm) circle (.5pt)(135:1cm) circle (.5pt)(180:1cm) circle (.5pt)(-135:1 cm) circle (.5pt)(-90:1 cm) circle (.5pt)(-45:1cm) circle (.5pt);
\filldraw($(180:1cm)! .5!(135:1cm)$) circle (.5pt)--(c)--($(0:1cm)! .5!(-45:1cm)$) circle (.5pt);
 \end{tikzpicture}
 \quad\in\tau^e
\end{equation}
The result reads
\begin{multline}
\pm u\left(\alpha\frac{d}{ae}\right)u\left(\beta\frac{c}{eb}\right)=\pm u\left(\gamma\frac{f}{ab}\right)\Leftrightarrow
\alpha\frac{d}{ae}+\beta\frac{c}{eb}=\gamma\frac{f}{ab}\\
\Leftrightarrow \alpha bd+\beta ac=\gamma ef
\end{multline}
Doing the same calculation for the short edge between the long edges $c$ and $d$, we also have
\begin{multline}
\pm u\left(\beta\frac{b}{ce}\right)u\left(\alpha\frac{a}{ed}\right)=\pm u\left(\delta\frac{f}{cd}\right)\Leftrightarrow
\beta\frac{b}{ce}+\alpha\frac{a}{ed}=\delta\frac{f}{cd}\\
\Leftrightarrow \beta bd+\alpha ac=\delta ef
\end{multline}
The two equalities can equivalently be rewritten as
\begin{equation}\label{eq:ptol}
 \alpha bd+\beta ac=\gamma ef,\quad \alpha\beta=\gamma\delta.
\end{equation}
This is exactly the signed Ptolemy relation of  \cite{MR2122727}.
\section{Proof of Theorem~\ref{thm:2}} \label{sec3}
\subsection{Part (i)}
Given $m\in\lambda_\alpha^{-1}(0)$. Let us choose $\tau\in\Delta^\alpha$. The signed Ptolemy transformation formula~\eqref{eq:ptol}  implies that 
\begin{equation}\label{eq:prod-1}
\prod_{t\in(\tau_\alpha)_2}\varepsilon_m(t)=-1,
\end{equation}
see \eqref{eq:epsm} for the definition of the function $\varepsilon_m$.
That means that one of the triangles of $\tau_\alpha$ is necessarily negative which proves part (i).
\subsection{Part (ii)}
Let $t\in(\tau_\alpha)_2$ be the negative triangle, and let  $(\tau_\alpha)_1= \{\alpha_\tau, a,b,c,d\}$ be arranged as in  this picture 
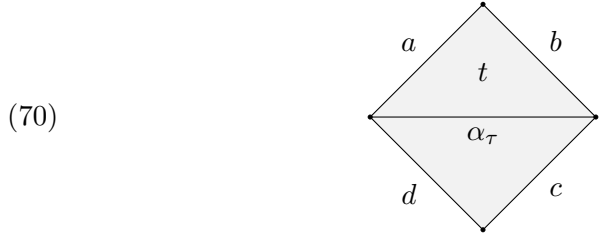
\begin{equation}
 \begin{tikzpicture}[scale=1.5,baseline=-3]
\filldraw[color=gray!10] (0:1cm)--(90:1 cm)--(180:1cm)--(-90:1cm)--cycle;
\node at (0,.4){$t$};
\draw [auto](180:1 cm)to node[swap] {$\alpha_\tau$}(0:1cm) 
(0:1cm)to node[swap] {$b$}(90:1cm)to node [swap]{$a$}(180:1cm)to node[swap] {$d$}(-90:1cm)to node[swap] {$c$}(0:1cm) ;
\filldraw (0:1cm) circle (.5pt)(90:1 cm) circle (.5pt)(180:1 cm) circle (.5pt)(-90:1cm) circle (.5pt);
 \end{tikzpicture}
\end{equation}
Apart from equality~\eqref{eq:prod-1}, the signed Ptolemy relation also implies that
\begin{equation}
\frac a d=\frac bc\equiv x
\end{equation}
which we can solve for $a$ and $b$:
\begin{equation}\label{eq:abmu}
a=x d,\quad b=x c.
\end{equation}
Here we assume that the sides of the quadrilateral $\tau_\alpha$ are geometrically pairwise distinct\footnote{It is possible that one pair of opposite sides of $\tau_\alpha$ are geometrically identical, for example, $b=d$. In that case, instead of \eqref{eq:abmu} one will have $a=x^2c$ and $b=d=cx$.}.
By substituting \eqref{eq:abmu} into $\psi_{\tau,t}$ and writing out explicitly the contributions from the quadrilateral $\tau_\alpha$, we calculate
\begin{multline}
\psi_{\tau,t}=-\frac{a}{b\alpha_\tau}-\frac{b}{a\alpha_\tau}-\frac{\alpha_\tau}{ab}+\frac{c}{d\alpha_\tau}+\frac{d}{c\alpha_\tau}+\frac{\alpha_\tau}{cd}+\sum_{s\in\tau_2\setminus(\tau_\alpha)_2}\frac{p_s}{q_s}\\
=-\frac{d}{c\alpha_\tau}-\frac{c}{d\alpha_\tau}-\frac{\alpha_\tau}{cdx^2}+\frac{c}{d\alpha_\tau}+\frac{d}{c\alpha_\tau}+\frac{\alpha_\tau}{cd}+\sum_{s\in\tau_2\setminus(\tau_\alpha)_2}\frac{p_s}{q_s}\\
=-\frac{\alpha_\tau}{cd}\left(x^{-2}-1\right)+\sum_{s\in\tau_2\setminus(\tau_\alpha)_2}\frac{p_s}{q_s}.
\end{multline}
The equality $\psi_{\tau,t}=0$ implies that $x<1$, and, as the other terms do not contain variable $\alpha_\tau$, we can  solve it explicitly for $\alpha_\tau$:
\begin{equation}\label{eq:alt}
\alpha_\tau=\frac{cd}{x^{-2}-1}\sum_{s\in\tau_2\setminus(\tau_\alpha)_2}\frac{p_s}{q_s}.
\end{equation}
Finally, by using the rules~\eqref{pic:assign}, we can calculate the parallel transport operator associated with $\alpha$ which happens to be represented by an upper triangular matrix with the diagonal elements being given by $\pm x^{\pm1}$ so that we get the relation
\begin{equation}
 x=e^{-\ell_\alpha(\phi(m))/2}.
\end{equation}

 \subsection{Part (iii)}
 If $\phi(m)=\phi(m')$ then, by choosing representatives $(\rho,h_0)$ and $(\rho',h_0)$ of $m$ and $m'$ respectively, we see that representations $\rho$ and $\rho'$ are conjugated by un upper triangular matrix (it must have $\infty$ as a fixed point) so that the equivalence~\eqref{eq:equiv} becomes evident.

\end{document}